\documentclass{amsart}
\usepackage{txfonts} 
\usepackage{amssymb,enumerate,bussproofs,url,hyperref}

\title{The logic of the reverse mathematics zoo}

\author{Giovanna D'Agostino}
 \address{Dipartimento di Matematica e Informatica,
    Universit\`{a} di Udine,
    viale delle Scienze 206,
    33100 Udine,
    Italy}
\email{giovanna.dagostino@uniud.it}

\author{Alberto Marcone}
 \address{Dipartimento di Matematica e Informatica,
    Universit\`{a} di Udine,
    viale delle Scienze 206,
    33100 Udine,
    Italy}
\email{alberto.marcone@uniud.it}
\urladdr{http://users.dimi.uniud.it/~alberto.marcone/}
\thanks{Marcone's research was supported by PRIN 2012 Grant ``Logica, Modelli e Insiemi''.\\
D'Agostino's research was supported by the GNCS-INDAM Project ``Algoritmica per il model checking e la sintesi di sistemi safety-critica''.}

\date{December 22, 2015}

\newtheorem{theorem}{Theorem}
\newtheorem{lemma}[theorem]{Lemma}
\newtheorem{definition}[theorem]{Definition}
\newtheorem{corollary}[theorem]{Corollary}

\theoremstyle{remark}
\newtheorem{example}[theorem]{Example}

\newcommand{\pl}{:\!\!-\ }

\begin{document}
\pagestyle{plain}
\begin{abstract}
Building on previous work by Mummert, Saadaoui and Sovine (\cite{Mummert}),
we study the logic underlying the web of implications and nonimplications
which constitute the so called reverse mathematics zoo. We introduce a
tableaux system for this logic and natural deduction systems for important
fragments of the language.
\end{abstract}

\maketitle

\tableofcontents

\section{Introduction}

Reverse mathematics is a wide ranging research program in the foundations of
mathematics: its goal is to systematically compare the strength of
mathematical theorems by establishing equivalences, implications and
nonimplications over a weak base theory. Currently, reverse mathematics is
carried out mostly in the context of subsystems of second-order arithmetic
and very often a specific system known as $\mathsf{RCA}_0$ is used as the
base theory.

The earlier reverse mathematics research, leading to Steve Simpson's
fundamental monograph \cite{sosoa}, highlighted the fact that most
mathematical theorems formalizable in second order arithmetic were in fact
either provable within $\mathsf{RCA}_0$ or equivalent to one of four other
specific subsystems, linearly ordered in terms of provability strength. This
is summarized by the \emph{Big Five} terminology coined by Antonio Montalb\'{a}n
in \cite{Mont}. However in recent years there has been a change in the
reverse mathematics main focus: following Seetapun's breakthrough result that
Ramsey theorem for pairs is not equivalent to any of the Big Five systems, a
plethora of statements, mostly in countable combinatorics, have been shown to
form a rich and complex web of implications and nonimplications. The first
paper featuring complex and non-linear diagrams representing statements of
second order arithmetics appears to be \cite{HS} (notice that the diagrams
appearing in \cite{wqo,interval} are of a different sort, as they deal with
properties of mathematical objects, rather than with mathematical
statements). Nowadays diagrams of this kind are a common feature of reverse
mathematics papers. This is called the zoo of reverse mathematics, a
terminology coined by Damir Dzhafarov when he designed \lq\lq a program to
help organize relations among various mathematical principles, particularly
those that fail to be equivalent to any of the big five subsystems of
second-order arithmetic\rq\rq. This program is available at \cite{rmzoo}.
Ludovic Patey's web site features a manually maintained zoo
(\cite{Pateyzoo}). The recent monograph \cite{Hirschfeldt}, devoted to a
small portion of the zoo, features a whole chapter of diagrams. These
diagrams cover also situations where a different base theory (e.g.\
$\mathsf{RCA}$, which is $\mathsf{RCA}_0$ with unrestricted induction) is
used, or where only the first order consequences are considered.

Actually, the zoo is not peculiar to subsystems of second order arithmetic.
For example, the study of weak forms of the Axiom of Choice and the
relationships between them has a long tradition in set theory: \cite{AC}
consists of a catalog of 383 forms of the Axiom of Choice and of their
equivalent statements. Connected to the book, there is also the web page
\cite{ACwww}, which claims also to be able to produce zoo-like tables;
unfortunately the site appears to be no longer maintained and, as of December
2015, the links are broken.

Mummert, Saadaoui and Sovine in \cite{Mummert} introduced a framework for
discussing the logic that is behind the web of implications and
nonimplications in the reverse mathematics zoo. They called their system
s-logic, introducing its syntax and semantics and proposing a tableaux system
for satisfiability of sets of s-formulas, and inference systems for two
fragments of s-logic (called $\mathcal{F}_1$ and $\mathcal{F}_2$, with the
first a subset of the second) that are important in the applications.

The present paper can be viewed as a continuation of \cite{Mummert}. Our goal
is to improve the systems introduced by Mummert, Saadaoui and Sovine and show
how widespread automated theorem proving tools can be used to deal
efficiently with s-logic. As a byproduct, our analysis also points out that,
notwithstanding the fact that the semantics for s-logic borrows some ideas
from the one for modal logic, s-logic is actually much closer to
propositional logic than to modal logic.

Here is the plan of the paper. After reviewing s-logic, in Section
\ref{sec:obs} we make some observations about its semantics. Using these, in
Section \ref{sec:tableau} we are able to simplify the tableaux system of
Mummert, Saadaoui and Sovine. Our formulation brings it closer to the
familiar tableaux systems for propositional logic, and thus, using an
efficient implementation of the latter, leads to more efficient algorithms.
Moreover, in Section \ref{sec:nd}, we improve also the treatment of the
fragments $\mathcal{F}_1$ and $\mathcal{F}_2$ by proposing natural deduction
systems for them. We also consider a new natural fragment of s-logic
$\mathcal{F}_3$, which includes $\mathcal{F}_2$ and for which we provide a
sound and complete natural deduction system. In Section \ref{sec:prolog} we
show how logical consequence between formulas of $\mathcal{F}_2$ (and hence
of $\mathcal{F}_1$) can be treated by using standard propositional Prolog:
this provides an efficient way of answering queries about whether a certain
implication or nonimplication follows from a database of known zoo facts.

\section{Basic observations about s-logic}\label{sec:obs}

For the reader's convenience, we start with a brief review of s-logic as
introduced in \cite{Mummert}.

We start from a set of propositional variables and we build propositional
formulas in the usual way, using the connectives $\neg$, $\land$, $\vee$,
and $\rightarrow$. An \emph{s-formula} is a formula of the form $A \strictif
B$ or $A \not\strictif B$, where $A$ and $B$ are propositional formulas. The
first type of s-formula is called positive or $\strictif$ s-formula, the
second one is negative or $\not\strictif$ s-formula. Notice that the
definition of s-formula is not recursive, and thus if $\alpha$ and $\beta$
are s-formulas neither $\alpha \land \beta$ nor $\alpha \strictif \beta$ are
s-formulas.

The intended meaning of $A \strictif B$ is that statement $A$ implies
statement $B$, over the fixed weak base theory. On the other hand $A
\not\strictif B$ asserts that $A \strictif B$ does not hold. In practice,
this happens when we have a model of the base theory in which $A$ holds and
$B$ does not (a counterexample to $A \strictif B$).

The semantics of s-logic is based on the notion of \emph{frame}, which is
just a nonempty set of valuations. Here by valuation we mean the usual notion
for propositional logic, i.e.\ a function assigning to every propositional
variable one of the truth values $T$ and $F$.

A frame $W$ \emph{satisfies} the positive s-formula $A \strictif B$ if for
every valuation $v \in W$ such that $v(A) = T$ we have also $v(B) = T$. $W$
satisfies the negative s-formula $A \not\strictif B$ if there exists a
valuation $v \in W$ such that $v(A) = T$ and $v(B) = F$.

Once we have the notion of satisfaction we can introduce in the usual way
notions such as \emph{satisfiability} of a set of s-formulas $\Gamma$ (there
exists a frame satisfying every member of $\Gamma$) and \emph{logical
consequence} between a set of s-formulas $\Gamma$ and a given s-formula
$\alpha$ (every frame satisfying $\Gamma$ satisfies also $\alpha$): for the
latter we use the notation $\Gamma \models_s \alpha$.

We point out that although $\rightarrow$ and $\strictif$ (and their
negations) are superficially similar, there are important difference between
them. For example, if $X$ and $Y$ are propositional variables, the set of
s-formulas $\{X \not\strictif Y, Y \not\strictif X\}$ is satisfiable (by a
frame with two valuations), while the \lq\lq corresponding\rq\rq\ set of
propositional formulas $\{\neg (X \rightarrow Y), \neg (Y \rightarrow X)\}$
is unsatisfiable. Expressing the same example in terms of logical
consequence, we have that although $\neg (X \rightarrow Y) \models Y
\rightarrow X$ in propositional logic, it is certainly not the case that $X
\not\strictif Y \models_s Y \strictif X$. Notice that in these examples we
are using s-formulas from $\mathcal{F}_1$.\smallskip

Mummert, Saadaoui and Sovine introduced also the following fragments of
s-logic:

\begin{definition}\label{f1f2}
The fragment $\mathcal{F}_1$, $\mathcal{F}_2$ of s-logic are:
\begin{itemize}
\item $\mathcal{F}_1$ is the set of all s-formulas of the forms $X\strictif
    Y$ and $X\not \strictif Y$, where  $X,Y$ are propositional variables;
\item $\mathcal{F}_2$ is the set of all s-formulas of the forms $A\strictif
    Y$ and $A\not \strictif Y$, where $A$ is a nonempty conjunction of
  propositional variables and $Y$ is a single propositional variable.
\end{itemize}
\end{definition}

As pointed out in \cite{Mummert}, $\mathcal{F}_1$ captures the basic
implications and nonimplications in reverse mathematics, while in
$\mathcal{F}_2$ we can express also results such as the equivalence between
Ramsey Theorem for pairs with two colors and the conjunction between the same
theorem restricted to stable colorings and the cohesiveness principle. Notice
that we do not need to consider also s-formulas with conjunctions of
propositional variables after $\strictif$, as $\Gamma \models_s A \strictif X
\land Y$ if and only if $\Gamma \models_s A \strictif X$ and $\Gamma
\models_s A \strictif Y$, while $\Gamma \models_s A \not\strictif X \land Y$
if and only if $\Gamma \models_s A \not\strictif X$ or $\Gamma \models_s A
\not\strictif Y$.

We introduce another fragment of s-logic, which is a natural generalization
of the fragment ${\mathcal F_2}$, and captures some implications between
members of the reverse mathematics zoo escaping ${\mathcal F_2}$. Recall, for
example, that the statement about the existence of iterates of continuous
mappings of the closed unit interval into itself was proved in \cite{FSY} to
be equivalent to the disjunction of weak K\"{o}nig's lemma and
$\mathbf{\Sigma}^0_2$-induction.

\begin{definition}\label{f3}
$\mathcal{F}_3$ is the set of all s-formulas of the forms $C \strictif D$ and
$C \not\strictif D$, where $C$ and $D$ are a nonempty conjunction of
propositional variables and a nonempty disjunction of propositional
variables, respectively.
\end{definition}

Here we do not need to consider also s-formulas with disjunctions of
propositional variables before $\strictif$, as $\Gamma \models_s X \lor Y
\strictif A$ if and only if $\Gamma \models_s X \strictif A$ and $\Gamma
\models_s Y \strictif A$, while $\Gamma \models_s X \lor Y \not\strictif A$
if and only if $\Gamma \models_s X \not\strictif A$ or $\Gamma \models_s Y
\not\strictif A$.\smallskip

We now make a couple of useful basic observations about the semantics of
s-logic which use the following definition.

\begin{definition}
Given a set of s-formulas $\Gamma$, the set of s-formulas $\Gamma^+,
\Gamma^-$ are defined as
\[
\Gamma^+:=\{A\strictif B: A\strictif B \in \Gamma\},\quad
\Gamma^-:=\{A\not \strictif B: A \not\strictif B \in \Gamma\},
\]
while $\Gamma^+_{prop}$ is the set of propositional formulas
\[
\Gamma^+_{prop}:=\{A\rightarrow B: A\strictif B\in \Gamma\}.
\]
\end{definition}

\begin{lemma}\label{lemma:sat}
Let $\Gamma$ be a set of s-formulas. The following are equivalent:
\begin{enumerate}
  \item $\Gamma$ is satisfiable;
  \item the set of s-formulas
\[
\Gamma^+ \cup \{A \not \strictif B\}
\]
is satisfiable, for each $A \not \strictif B \in \Gamma^-$;
  \item the set of propositional formulas
\[
\Gamma^+_{prop} \cup \{A, \neg B\}
\]
is satisfiable (in the usual sense of propositional logic), for each $A
\not \strictif B \in \Gamma^-$.
\end{enumerate}
\end{lemma}
\begin{proof}
(1) implies (2) is immediate.

To prove that (2) implies (3) fix $A \not\strictif B \in \Gamma^-$. Since
$\Gamma^+ \cup \{A \not\strictif B\}$ is satisfiable, there exists a frame
$W$ which validates this set of s-formulas; hence there exists a valuation $v
\in W$ with $v(A)=T$, $v(B)=F$. Since $W \models X \strictif Y $ for all $X
\strictif Y \in \Gamma^+$ we have that $v(X) = T$ implies $v(Y) =T$ for each
such s-formula. Hence $v$ satisfies the set of propositional formulas
$\Gamma^+_{prop} \cup \{ A, \neg B\}$.

For (3) implies (1), suppose (3) holds, and for each $A \not \strictif B \in
\Gamma^- $ let $w_{A \not \strictif B}$ be a valuation satisfying the set of
propositional formulas $\Gamma^+_{prop} \cup \{ A, \neg B\}$. Let $W$ be the
frame consisting of all these valuations: $W=\{w_{A \not \strictif B}: A \not
\strictif B \in \Gamma\}$. It is easily seen that $W$ satisfies $\Gamma$.
\end{proof}

\begin{corollary}\label{cor:unsat}
A set of s-formulas $\Gamma$ is unsatisfiable if and only if there exists $A
\not \strictif B\in \Gamma^-$ such that $\Gamma^+\models A \strictif B$. In
particular, every set of positive s-formulas is satisfiable.
\end{corollary}

Lemma \ref{lemma:sat} suggests a fairly simple algorithm for the
satisfiability problem for sets of s-formulas. In fact given the set of
s-formulas $\Gamma$ one needs only to check whether for each $A \not
\strictif B \in \Gamma^-$ the set of propositional formulas $\Gamma^+_{prop}
\cup \{A, \neg B\}$ is satisfiable. Given the constant improvement in the
efficiency of SAT-solvers (see e.g. \cite{SAT1,SAT2}), this is in fact a
quite efficient way of dealing with the problem.

\begin{corollary}
The problem of satisfiability for a (finite) set of s-formulas has the same
complexity of propositional satisfiability, i.e.\ it is NP-complete.
\end{corollary}
\begin{proof}
The problem is in NP because, if we fix a finite set of s-formulas $\Gamma$
and set $n = |\Gamma|$ and $k = |\Gamma^+|$, using the last point of the
previous Lemma, we can reduce the satisfiability of $\Gamma$ to the
satisfiability of $n-k$ sets of propositional formulas each of cardinality
$k+1$.

The problem is NP-complete because it essentially contains propositional
satisfiability.
\end{proof}

The previous corollary implies that with respect to complexity s-logic is
more similar to propositional logic than to modal logic (recall that
satisfiability for propositional logic is NP-complete, while satisfiability
for the modal logic K is PSPACE-complete).

Next, we consider logical consequence among s-formulas.

\begin{lemma}\label{lemma:logcons}
Let $\Gamma$ be a satisfiable set of s-formulas. For any propositional
formulas $A$ and $B$ we have:
\begin{enumerate}[(i)]
  \item $\Gamma \models_s A \strictif B$ if and only if  $\Gamma^+
      \models_s A \strictif B$ if and only if $\Gamma^+_{prop} \models A
      \rightarrow B$;
  \item $\Gamma \models_s A \not\strictif B$ if and only if there exists an
      s-formula $E \not\strictif F \in \Gamma^-$ such that
  \[\Gamma^+, A \strictif B \models_s E  \strictif F,\]
      if and only if there exists an s-formula $E \not\strictif F \in
      \Gamma^-$ such that
      \[\Gamma^+_{prop}, A \rightarrow B \models E \rightarrow F.\]
\end{enumerate}
\end{lemma}
\begin{proof}
$(i)$ If $\Gamma \models_s A \strictif B$ then $\Gamma \cup \{A \not\strictif
B\}$ is unsatisfiable. By Lemma \ref{lemma:sat}, there exists $E
\not\strictif F \in \Gamma^-\cup\{A \not \strictif B\}$ such that
$\Gamma^+\cup \{ E \not\strictif F\}$ is unsatisfiable. Since $\Gamma$ is
satisfiable, $E \not\strictif F$ must be $A \not\strictif B$, and hence
$\Gamma^+\models_s A \strictif B$. The viceversa is obvious. The equivalence
between $\Gamma^+\models_s A \strictif B$ and $\Gamma^+_{prop} \models A
\rightarrow B$ follows easily from Lemma \ref{lemma:sat}.

As for $(ii)$, $\Gamma \models_s A \not\strictif B$ iff the set of s-formulas
$\Gamma \cup\{A \strictif B\}$ is unsatisfiable iff (by Lemma
\ref{lemma:sat}) there exists $E\not \strictif F \in \Gamma^-$ such that
$\Gamma^+ \cup\{A \strictif B\} \cup\{E\not \strictif F\}$ is unsatisfiable
iff there exists $E\not \strictif F \in \Gamma^-$ such that $\Gamma^+, A
\strictif B \models_s E  \strictif F$ iff $\Gamma^+_{prop}, A \rightarrow B
\models E \rightarrow F$.
\end{proof}

The previous Lemma says that only positive s-formulas are needed to check
whether a positive s-formula is logical consequence of a satisfiable set of
s-formulas. Moreover, if only positive s-formulas are considered, their logic
does not differ substantially from propositional logic, because $\strictif$
behaves exactly as $\rightarrow$.

If we want to prove that a negative s-formula is logical consequence of a
satisfiable set of s-formulas then differences with propositional logic do
appear. The previous Lemma tells us that the collection of $\not\strictif$
s-formulas which are logical consequences of some $\not\strictif$ s-formulas
(i.e.\ typically from the existence of different models showing that the
implications fail) and some $\strictif$ s-formulas is just the union of the
consequences of a single $\not\strictif$ s-formula and the given set of
$\strictif$ s-formulas. In other words, having two models available
simultaneously gives no new information. This might again suggest that
s-logic is not substantially different from propositional logic.
Nevertheless, the deductive meta-properties of s-logic and propositional
logic differ, as showed by the following example.

\begin{example}\label{ex1}
In propositional logic, if $A,B,C,D$ are propositional variables and $\alpha$
is a formula, we have: \[\Gamma, A\rightarrow C \models \alpha \quad \hbox{\
and\ } \quad \Gamma, B\rightarrow C \models \alpha \quad \hbox{\ then \ }
\quad \Gamma, A\land B \rightarrow C\models \alpha.\]

This is not the case in s-logic because, for example: \vskip5pt
\[A\not \strictif D, B\not \strictif D, A \strictif C \models_s C\not \strictif D,\]
\[A\not \strictif D, B\not \strictif D,  B \strictif C\models_s C\not \strictif D\]
but
\[
A \not\strictif D, B \not\strictif D, A \land B \strictif C \nvDash_s C \not\strictif D.
\]
In fact the set of s-formulas $\{A \not\strictif D, B \not\strictif D, A
\land B \strictif C, C \strictif D\}$ is satisfied e.g.\ by the frame $W =
\{v_1,v_2\}$ with $v_1(A) = v_2(B) = T$, $v_1(B) = v_2(A) = v_1(D) = v_2(D) =
v_1(C) = v_2(C) = F$.
\end{example}

\section{Tableaux for s-logic}\label{sec:tableau}

Another application of Lemma \ref{lemma:sat} regards the existence of a
tableaux system to check unsatisfiability of finite set of s-formulas. In
\cite{Mummert}, the authors introduce a tableaux system which keeps track of
valuations in the syntax. For this reason the tableaux are unusual compared
e.g.\ to the standard tableaux described in a textbook such as \cite{BenAri}
(see \S2.6, where they are called semantic tableaux). In fact to deal with
strict non-implication the system considers not only s-formulas, but also
so-called {\em world formulas}, that is, pairs $(A,v)$ where $A$ is a
propositional formula and $v$ represents a variable for a propositional
evaluation. The tableaux system of \cite{Mummert} contains e.g.\ the
following rule (where $\Gamma$ is a set of s- and world formulas, and $v$ is
new for $\Gamma$):\footnote{here and below we adopt the convention that the
premisses of a rule are above their consequence, while in \cite{Mummert} the
reverse convention is adopted}
\[
{{\Gamma, A \not\strictif B }
\over
{\Gamma, (A,v), (\neg B,v)} }
\]
The tableaux system of \cite{Mummert} has also the peculiarity of not
discharging the formulas which are used in a step (this is instead a common
feature of tableaux systems for propositional logic, see \cite[Algorithm
2.64]{BenAri}). This is motivated by the fact that positive s-formulas are in
fact universal assertions about the collection of all possible worlds, and
thus might be used again on a different world. However Lemma \ref{lemma:sat}
shows that this precaution is superfluous, because the unsatisfiability of a
set of s-formulas depends only on a single world, the one witnessing the
satisfiability of one of the negative s-formulas that imply the
unsatisfiability of the whole set.

A straightforward application of Lemma \ref{lemma:sat} leads to a more
traditional tableaux system, which has the advantage of dealing only with
propositional formulas, except for the first (root) step. This system can be
described as follows. The rules of the system are given by the standard rules
of a traditional tableaux system for propositional logic plus the {\em
$\not\strictif$-rule}, which is:
\[
{{\Gamma, A\not \strictif B}
\over
{\Gamma^+_{prop}, A, \neg B}},
\]
subsuming the rule
\[
{{\Gamma}
\over
{\Gamma^+_{prop}}}
\]
when $\Gamma^- = \emptyset$.

Notice that, starting from $\Gamma, A \not\strictif B, C \not\strictif D$,
the $\not\strictif$-rule allows to derive either $\Gamma^+_{prop}, A, \neg B$
or $\Gamma^+_{prop}, C, \neg D$.

\begin{definition}
A tableau for a set of s-formulas $\Gamma$ is a finite tree $T$ such that:
\begin{enumerate}[(a)]
 \item  the root of $T$ is labeled by $\Gamma$, while the inner nodes are
     labelled by sets of propositional formulas;
\item the label of the child of the root is obtained from the label of the
    root by an application of the $\not\strictif$-rule;
\item the label of every other node is obtained from the label of its
    parent by one of the standard propositional tableaux inference rules
    (see e.g.\ \cite[Algorithm 2.64]{BenAri}).
\end{enumerate}
A path through a tableau is closed if it contains a node for which the label
contains both $A$ and $\neg A$ for some propositional formula A. A tableau is
closed if every maximal branch is closed.
\end{definition}

Notice that, in contrast with the propositional case, a given set of
s-formulas might have both closed and non-closed tableaux. In fact to obtain
a closed tableau we must pick the \lq\lq right\rq\rq\ negative s-formula when
we apply the $\not\strictif$-rule to construct the child of the root, as is
easily seen for the set of s-formulas $\{A \not\strictif A, A \not\strictif
B\}$.

Applying Lemma \ref{lemma:sat} we immediately obtain:

\begin{corollary}
A set of s-formula $\Gamma$ is unsatisfiable if and only if there exists a
tableau for $\Gamma$ in which every branch is closed.
\end{corollary}

The previous corollary is useful in practice, because to check satisfiability
of s-formulas after the first step we use a standard tableaux system for
propositional logic.

However, the tableaux system presented here and the one proposed in
\cite{Mummert} are hybrid systems, where s-formulas and propositional
formulas coexist. Hence neither system is appropriate to study s-logic for
itself, and compare its deductive properties with the ones of propositional
logic, as we did in Example \ref{ex1}. What are the rules of s-logic, and can
we have a calculus dealing exclusively with s-formulas? As in \cite{Mummert},
we answer these questions for some fragments of s-logic which are relevant to
the practice of reverse mathematics. In our case these are the ones
introduced in Definition \ref{f1f2} (considered also in \cite{Mummert}) but
also the fragment $\mathcal{F}_3$ introduced in Definition \ref{f3}.

\section{Natural deductions for fragments of s-logic}\label{sec:nd}
Lemma \ref{lemma:logcons} is especially useful when dealing with the
fragments of Definitions \ref{f1f2} and \ref{f3}. In \cite{Mummert} sound and
complete deductive systems for $\mathcal{F}_1$ and $\mathcal{F}_2$ are
presented.

The system for $\mathcal{F}_1$ consists of the following axioms and rules:
\begin{description}
\item[(Axiom)] $X \strictif X$, where $X$ is a propositional variable;
\item[(HS)] From $A \strictif X$ and $ X \strictif Y$ deduce $A \strictif
    Y$;
\item[(N)] From $X \not \strictif Y$, $X \strictif W$ and $Z \strictif Y$
    deduce $W\not \strictif Z$.
\end{description}

The system for $\mathcal{F}_2$ consists of the following axioms and rules:
\begin{description}
\item[(Axiom)] $X\strictif X$, where $X$ is a propositional variable;
\item[(W)] From $A \strictif Y$, deduce $B \strictif Y$, where $B$ is any
    conjunction such that every conjunct of $A$ is also a conjunct of $B$;
\item[(HS)] From $X \land B \strictif Y$ and $A \strictif X$, deduce $A
    \land B \strictif Y$;
\item[(N)] From $A \not\strictif X$, $A \land Z \strictif X$, and $A
    \strictif Y$ for each conjunct $Y$ of $B$, deduce $B \not\strictif Z$.
\end{description}

We propose natural deduction calculi for ${\mathcal F_2}$ and for ${\mathcal
F_1}$, differing from the systems in \cite{Mummert} because of a simpler rule
for negative s-formulas. We also introduce a natural deduction system for
${\mathcal F_3}$. These systems are presented in a style similar to
\cite{HuthRyan} (see \S1.2.3 for a summary of natural deduction for
propositional logic).

\subsection{A Natural Deduction Calculus for ${\mathcal F_2}$}\label{F2}

The Natural Deduction Calculus for ${\mathcal F_2}$ has the following axioms
and rules, where $X,Y,Z,X_i, \dots$ are propositional variables, $A,B,C,
\dots$ are arbitrary (possibly empty) conjunctions of propositional
variables, $\alpha$ is an arbitrary ${\mathcal F_2}$ formula, and $\Gamma$
and $\Gamma'$ are sets of ${\mathcal F_2}$ s-formulas:

\[
\hbox{\textbf{(Axiom)}:}\quad X \strictif X
\]
\bigskip

\begin{prooftree}
\AxiomC{$\Gamma$}
\noLine
\UnaryInfC{$\triangledown$}
\noLine
\UnaryInfC{$A\strictif Y$}
\LeftLabel{\textbf{(conj1)}:\quad}
\UnaryInfC{$A\land B\strictif Y$}
\end{prooftree}
\bigskip

\begin{prooftree}
\AxiomC{$\Gamma$}
\noLine
\UnaryInfC{$\triangledown$}
\noLine
\UnaryInfC{$X_1 \land \ldots \land X_n \strictif Y$}
\LeftLabel{\textbf{(conj2)}:\quad}
\UnaryInfC{$X_{i_1} \land \ldots \land X_{i_k} \strictif Y$,}
\end{prooftree}
where $\{X_{i_1}, \dots, X_{i_k}\} = \{X_1, \dots, X_n\}$ as {\em sets} of
propositional variables.\bigskip

\begin{prooftree}
\AxiomC{$\Gamma$}
\noLine
\UnaryInfC{$\triangledown$}
\noLine
\UnaryInfC{$A \strictif Y$}
\AxiomC{$\Gamma'$}
\noLine
\UnaryInfC{$\triangledown$}
\noLine
\UnaryInfC{$Y\land B\strictif Z$}
\LeftLabel{\textbf{(trans)}:\quad}
\BinaryInfC{$A\land B\strictif Z$}
\end{prooftree}
\bigskip

 \begin{prooftree}
 \AxiomC{$\Gamma$ }
\noLine
\UnaryInfC{$\nabla$}
\noLine
\UnaryInfC{$A\strictif B$ }
 \AxiomC{$\Gamma$ }
\noLine
\UnaryInfC{$\nabla$}
\noLine
\UnaryInfC{$A\not\strictif B$ }
\LeftLabel{\textbf{($\bot$)}:\quad}
\BinaryInfC{$\alpha$ }
\end{prooftree}

For negative s-formulas we want a rule allowing to construct of a proof of $A
\not\strictif X$ from hypothesis $\Gamma, C \not\strictif Y$, whenever we
have a proof of $C \strictif Y$ from hypothesis $\Gamma, A \strictif X$:
\begin{prooftree}
\AxiomC{$\Gamma'$}
\noLine
\UnaryInfC{$\triangledown$}
\noLine
\UnaryInfC{$C \not\strictif Y$}
\AxiomC{$\Gamma, [A \strictif X]$}
\noLine
\UnaryInfC{$\triangledown$}
\noLine
\UnaryInfC{$C \strictif Y$}
\LeftLabel{\textbf{(neg)}:\quad}
\BinaryInfC{$A \not\strictif X$}
\end{prooftree}
\bigskip

Let $\Gamma \rhd_{\mathcal F_2} \alpha$ denote the existence of a natural
deduction proof (in the system just described) of the ${\mathcal F_2}$
s-formula $\alpha$ from hypothesis in the set of ${\mathcal F_2}$ s-formulas
$\Gamma$.

\begin{example}\label{ex:N2}
Here is a deduction showing that
\[
A \not\strictif X, A \land Z \strictif X, A \strictif Y_1, \ldots, A \strictif Y_n
\rhd_{\mathcal F_2} Y_1 \land \dots \land Y_n \not\strictif Z,
\]
corresponding to rule (N) in the ${\mathcal F_2}$ system of \cite{Mummert}:
\begin{center}
\begin{prooftree}
\AxiomC{$A \not \strictif X$} \AxiomC{$A \strictif Y_n$ } \AxiomC{$A \strictif Y_2$ } \AxiomC{$A \strictif Y_1$ }
\AxiomC{$[Y_1\land \ldots\land Y_n\strictif Z] $}
\BinaryInfC{$A \land Y_2 \land \dots \land Y_n \strictif Z$}
\doubleLine
\BinaryInfC{$A \land Y_3 \land \dots \land Y_n \strictif Z$}
\noLine
\UnaryInfC{\vdots}
\noLine
\UnaryInfC{$A \land Y_n \strictif Z$}
\doubleLine
\BinaryInfC{$A \strictif Z$}
\AxiomC{$A \land Z \strictif X$}
\doubleLine
\BinaryInfC{$A \strictif X$}
\BinaryInfC{$Y_1 \land \ldots \land Y_n \not\strictif Z$}
\end{prooftree}
\end{center}
Here double lines indicate combined applications of (conj2) and (trans), the
top step consists of an application of (trans), and the last step is an
application of (neg).
\end{example}

One can easily prove that all ${\mathcal F_2}$ rules are sound with respect
to s-logical consequence. As for completeness, we divide the proof into
cases, depending on the satisfiability of the set of premisses $\Gamma$.

\begin{lemma}\label{lem:F2positive}
If $\Gamma$ is a satisfiable set of ${\mathcal F_2}$ s-formulas and $\alpha$
is a ${\mathcal F_2}$ s-formula such that $\Gamma \models_s \alpha$ then
$\Gamma \rhd_{\mathcal F_2} \alpha$.
\end{lemma}
\begin{proof}
To prove the Lemma we rely on Theorem 17 from \cite{Mummert}, which says that
if $\Gamma$ is a satisfiable\footnote{actually, the hypothesis in
\cite{Mummert} is that $\Gamma$ is consistent, but an inspection of the proof
reveals that the right hypothesis is the one of satisfiability.} and
$\Gamma\models_s \alpha$ then $\alpha$ is derivable from $\Gamma$ using the
rules (Axiom), (W), (HS), and (N). Hence, to show that $\alpha$ is derivable
in our system it is enough to show the existence of natural deduction proofs
for rules (W), (HS), and (N). The only nontrivial case is rule (N), which is
dealt with in Example \ref{ex:N2}.
\end{proof}

To finish the completeness proof for $\rhd_{\mathcal F_2}$, we have to
consider the case when $\Gamma$ is unsatisfiable, where we need to prove that
$\Gamma \rhd_{\mathcal F_2} \alpha$, for any ${\mathcal F}_2$ s-formula
$\alpha$.

\begin{lemma}\label{lem:unsat}
If $\Gamma$ is unsatisfiable, then for any ${\mathcal F_2}$ s-formula $\alpha$
we have $\Gamma \rhd_{\mathcal F_2} \alpha$.
\end{lemma}
\begin{proof}
By Corollary \ref{cor:unsat}, if $\Gamma$ is unsatisfiable then there exists
$A \not \strictif B \in \Gamma^-$ such that $\Gamma^+ \models_s A \strictif
B$. Since $\Gamma^+$ is satisfiable (again by Corollary \ref{cor:unsat}), by
Lemma \ref{lem:F2positive} we have $\Gamma^+ \rhd_{\mathcal F_2} A \strictif
B$. Hence $\Gamma \rhd_{\mathcal F_2} A \strictif B$, and $\Gamma
\rhd_{\mathcal F_2} \alpha$ follows by rule ($\bot$).
\end{proof}

Putting all the results of this subsection together, we obtain:

\begin{theorem}
If $\Gamma$ is a set of ${\mathcal F}_2$ s-formulas and $\alpha$ is a
${\mathcal F}_2$ s-formula, then
\[
\Gamma \models_s \alpha \qquad \Leftrightarrow \qquad \Gamma \rhd_{\mathcal F_2} \alpha.
\]
\end{theorem}

\subsection{A Natural Deduction Calculus for ${\mathcal F_1}$}

The Natural Deduction Calculus for ${\mathcal F_1}$ has the following axioms
and rules (where $X,Y,Z$ are propositional variables, $\alpha$ is a
${\mathcal F_1}$ s-formula, and $\Gamma$ and $\Gamma'$ are sets of ${\mathcal
F_1}$ s-formulas):

\[
\hbox{\textbf{(Axiom)}:}\quad X \strictif X
\]
\bigskip

\begin{prooftree}
\AxiomC{$\Gamma$}
\noLine
\UnaryInfC{$\triangledown$}
\noLine
\UnaryInfC{$X\strictif Y$}
\AxiomC{$\Gamma'$ }
\noLine
\UnaryInfC{$\triangledown$}
\noLine
\UnaryInfC{$Y\strictif Z$}
\LeftLabel{\textbf{(trans)}:\quad}
\BinaryInfC{$X\strictif Z$}
\end{prooftree}
\bigskip

\begin{prooftree}
\AxiomC{$\Gamma'$}
\noLine
\UnaryInfC{$\triangledown$}
\noLine
\UnaryInfC{$Y \not\strictif Z$}
                                                  \AxiomC{$\Gamma, [X \strictif Y]$}
                                                  \noLine
                                                  \UnaryInfC{$\triangledown$}
                                                  \noLine
                                                  \UnaryInfC{$Y \strictif Z$}
\LeftLabel{\textbf{(neg)}:\quad} \BinaryInfC{$X \not\strictif Y$}
\end{prooftree}
\bigskip

 \begin{prooftree}
 \AxiomC{$\Gamma$ }
\noLine
\UnaryInfC{$\nabla$}
\noLine
\UnaryInfC{$A\strictif B$ }
 \AxiomC{$\Gamma'$}
\noLine
\UnaryInfC{$\nabla$}
\noLine
\UnaryInfC{$A\not\strictif B$ }
\LeftLabel{\textbf{($\bot$)}:\quad}
\BinaryInfC{$\alpha$ }
\end{prooftree}
\bigskip

Let $\Gamma \rhd_{\mathcal F_1}  \alpha$ denotes the existence of a natural
deduction proof (in the system just described) of the ${\mathcal F_1}$
s-formula $\alpha$ from hypothesis in the set of ${\mathcal F_1}$ s-formulas
$\Gamma$.

\begin{example}\label{ex:N1}
Here is a deduction showing that
\[
X \not\strictif Y, X \strictif W, Z \strictif Y
\rhd_{\mathcal F_1} W \not\strictif Z,
\]
corresponding to rule (N) in the ${\mathcal F_1}$ system of \cite{Mummert}:
\begin{center}
\begin{prooftree}
\AxiomC{$X \strictif W$} \AxiomC{$[W \strictif Z]$ }
                \BinaryInfC{$X \strictif Z$}         \AxiomC{$Z \strictif Y $ }
                                \BinaryInfC{$X \strictif Y$}         \AxiomC{$X \not \strictif Y$ }
                                             \BinaryInfC{$W\not\strictif Z$}
\end{prooftree}
\end{center}
Here we employed (trans) twice and (neg) for the last step.
\end{example}

As for the case of the ${\mathcal F_2}$ system, the soundness of
$\rhd_{\mathcal F_1}$ is easily proved, and left to the reader. For
completeness, we may follow the same line of the completeness proof for
$\rhd_{\mathcal F_2}$, dividing the proof into cases, depending on whether
$\Gamma$ is a satisfiable set of ${\mathcal F_1}$ s-formulas or not. The case
where $\Gamma$ is satisfiable can be dealt using Theorem 20 from
\cite{Mummert}, and consists in proving the ${\mathcal F_1}$ rules of
\cite{Mummert} in our system. The only nontrivial case is rule (N), which is
dealt with in Example \ref{ex:N1}. In the case where $\Gamma$ is
unsatisfiable, we may proceed using rule $\bot$ as we did for $\rhd_{\mathcal
F_2}$. Hence:

\begin{theorem}
If $\Gamma$ is a set of ${\mathcal F}_1$ s-formulas and $\alpha$ is a
${\mathcal F}_1$ s-formula, then
\[
\Gamma\models_s \alpha \qquad \Leftrightarrow \qquad \Gamma \rhd_{\mathcal F_1} \alpha.
\]
\end{theorem}

\subsection{A Natural Deduction Calculus for ${\mathcal F_3}$}

We now consider the fragment $\mathcal F_3$ introduced in Definition
\ref{f3}. In considering an $\mathcal F_3$ s-formula $C\strictif D$ or $C\not
\strictif D$ we denote by $C_i$ a propositional variable which is a
$C$-conjunct and by $D_j$ a propositional variable which is a $D$-disjunct.

In order to capture derivability in fragment $\mathcal F_3$, we extend our
natural deduction calculus for $\mathcal F_2$ with the following two rules:

\begin{prooftree}
\AxiomC{$\Gamma$}
\noLine
\UnaryInfC{$\triangledown$}
\noLine
\UnaryInfC{$A\strictif B$}
\LeftLabel{\textbf{(disj1)}:\quad}
\UnaryInfC{$A\strictif D$}
\end{prooftree}
where $\{B_{1}, \dots, B_{n}\} \subseteq  \{D_1, \dots, D_h\}$ as {\em sets} of
propositional variables.
\bigskip

\begin{prooftree}
 \AxiomC{$\Gamma$ }
\noLine
\UnaryInfC{$\nabla$}
\noLine
\UnaryInfC{$A\strictif B$ }
 \AxiomC{$\Gamma, [A\strictif B_1]$ }
\noLine
\UnaryInfC{$\nabla$}
\noLine
\UnaryInfC{$C\strictif E$}
\AxiomC{$\ldots$}
\noLine
\UnaryInfC{$ $}
 \AxiomC{$\Gamma, [A\strictif B_n]$ }
\noLine
\UnaryInfC{$\nabla$}
\noLine
\UnaryInfC{$C\strictif E$}
\LeftLabel{\textbf{(disj2)}:\quad }
\QuaternaryInfC{$C\strictif E$}
\end{prooftree}
where $B= B_1 \lor \dots \lor B_n$ and $\Gamma$ is a set of positive
s-formulas.

\begin{lemma}
Rules (disj1) and (disj2) are sound in s-logic.
\end{lemma}

\begin{proof}
Soundness of rule (disj1) is immediate.

As for rule (disj2), suppose $\Gamma$ is a positive set of $\mathcal F_3$
s-formulas and $B= B_1 \lor \dots \lor B_n$ is such that:
\begin{itemize}
  \item $\Gamma \models_s A \strictif B$;
  \item $\Gamma, A \strictif B_i \models_s C\strictif E$ for each $i = 1,
      \dots, n$.
\end{itemize}
We want to prove that $\Gamma \models_s C \strictif E$. Since $\Gamma$
contains only positive s-formulas, by Corollary \ref{cor:unsat} each set
$\Gamma, A \strictif B_i$ is satisfiable. Hence we may apply Lemma
\ref{lemma:logcons} obtaining:
\[
\Gamma^+_{prop}, A \rightarrow B_i \models C \rightarrow E.
\]
Similarly we obtain $\Gamma^+_{prop} \models A \rightarrow B$, that is
$\Gamma^+_{prop}, A \models B$. By propositional reasoning it follows that
$\Gamma^+_{prop} \models C \rightarrow E$. Hence, by Lemma
\ref{lemma:logcons} again, $\Gamma \models_s C \strictif E$.
\end{proof}

Notice that the restriction to positive set of s-formulas $\Gamma$ in rule
(disj2) is necessary because without this hypothesis the rule is no longer
sound. To see this consider e.g.\ the set
\[
\Gamma=\{A \strictif B_1 \lor B_2, A \not\strictif B_1, A \not\strictif B_2 \}.
\]
$\Gamma$ is satisfiable, while each set $\Gamma \cup \{ A \strictif B_i\}$,
for $i=1,2$, is unsatisfiable. It follows that any formula $C \strictif D$
(with $C, D$ new for $\Gamma$) is a s-consequence of both sets $\Gamma \cup
\{ A \strictif B_i\}$. Moreover, $\Gamma \models_s A \strictif B_1 \lor B_2$,
but $C \strictif D$ is not a s-consequence of $\Gamma$.
\medskip

We denote $\mathcal F_3$-derivability by $\rhd_{\mathcal F_3}$. In proving
the completeness of the $\mathcal F_3$ system we shall use also the following
three rules, that will be shown to be derivable in our system in the next
Lemma.

\begin{prooftree}
 \AxiomC{$\Gamma$ }
\noLine
\UnaryInfC{$\nabla$}
\noLine
\UnaryInfC{$B\strictif A$ }
 \AxiomC{$\Gamma $ }
\noLine
\UnaryInfC{$\nabla$}
\noLine
\UnaryInfC{$C\strictif B_1$}
\AxiomC{$\ldots$}
\noLine
\UnaryInfC{$ $}
 \AxiomC{$\Gamma $ }
\noLine
\UnaryInfC{$\nabla$}
\noLine
\UnaryInfC{$C\strictif B_n$}
\LeftLabel{\textbf{($\mathbf{r_2}$)}:\quad }
\QuaternaryInfC{$C\strictif A$}
\end{prooftree}
where $B = B_1 \land \dots \land B_n$.\bigskip

\begin{prooftree}
 \AxiomC{$\Gamma$ }
\noLine
\UnaryInfC{$\nabla$}
\noLine
\UnaryInfC{$D\strictif E$ }
 \AxiomC{$\Gamma $ }
\noLine
\UnaryInfC{$\nabla$}
\noLine
\UnaryInfC{$D\land E_1\strictif F$}
\AxiomC{$\ldots$}
\noLine
\UnaryInfC{$ $}
 \AxiomC{$\Gamma $}
\noLine
\UnaryInfC{$\nabla$}
\noLine
\UnaryInfC{$D \land E_n\strictif F$}
\LeftLabel{\textbf{($\mathbf{r_3}$)}:\quad }
\QuaternaryInfC{$D\strictif F$}
\end{prooftree}
where $E = E_1 \lor \dots \lor E_n$.\bigskip

\begin{prooftree}
 \AxiomC{$\Gamma\qquad \qquad \qquad \Gamma$ }
\noLine
\UnaryInfC{$\nabla \qquad \ldots  \qquad \nabla$}
\noLine
\UnaryInfC{$A^1\strictif B^1 \qquad \qquad A^n\strictif B^n $ }
\AxiomC{$\ldots$}
\noLine
\UnaryInfC{$ $}
\AxiomC{$\Gamma, [A^1\strictif B^1_{h_1},\ldots,  A^n\strictif B^n_{h_n} ]$ }
\noLine
\UnaryInfC{$\nabla$}
\noLine
\UnaryInfC{$C\strictif E$}
\AxiomC{$\ldots$}
\noLine
\UnaryInfC{$ $}
\LeftLabel{\textbf{(disj2gen)}:\quad }
\QuaternaryInfC{$C\strictif E$}
\end{prooftree}
In (disj2gen), we require $\Gamma$ to be a set of positive s-formulas, and we
have a premise
\begin{prooftree}
\AxiomC{$\Gamma, A^1\strictif B^1_{h_1},\ldots,  A^n\strictif B^n_{h_n}$ }
\noLine
\UnaryInfC{$\nabla$}
\noLine
\UnaryInfC{$C\strictif E$}
\end{prooftree}
for every choice of indices $h_1, \dots, h_n$ such that $B^i_{h_i}$ is a
disjunct of $B^i$.

\begin{lemma}
($r_2$) is a derived rule in the $\mathcal F_2$ system, while ($r_3$) and
(disj2gen) are derived rules in the $\mathcal F_3$ system.
\end{lemma}
\begin{proof}
First, we provide a proof for  $(r_2)$ in the $\mathcal F_2$ system.
\begin{prooftree}
 \AxiomC{$\Gamma$ }
\noLine
\UnaryInfC{$\nabla$}
\noLine
\UnaryInfC{$B_1\land\ldots\land B_n\strictif A$ }
 \AxiomC{$\Gamma $ }
\noLine
\UnaryInfC{$\nabla$}
\noLine
\UnaryInfC{$C\strictif B_1$}
\BinaryInfC{$C\land B_2\land \ldots \land B_n\strictif A$}
\AxiomC{$\Gamma $ }
\noLine
\UnaryInfC{$\nabla$}
\noLine
\UnaryInfC{$C\strictif B_2$}
\doubleLine
\BinaryInfC{$C\land B_3\land \ldots \land B_n\strictif A$}
\noLine
\UnaryInfC{$\vdots$}
\UnaryInfC{$C\land  B_n\strictif A$}
 \AxiomC{$\Gamma $ }
\noLine
\UnaryInfC{$\nabla$}
\noLine
\UnaryInfC{$C\strictif B_n$}
\doubleLine
\BinaryInfC{$C\strictif A$}
\end{prooftree}
where in the first step we apply (trans) and then, in correspondence of each
double line, we use a combination of applications of (trans) and (conj2).
\bigskip

We now show how to derive $(r_3)$ in the $\mathcal F_3$ system.
\begin{prooftree}
 \AxiomC{$\Gamma$ }
\noLine
\UnaryInfC{$\nabla$}
\noLine
\UnaryInfC{$D\strictif E$ }
 \AxiomC{$[D\strictif E_1]$ }
 \AxiomC{$\Gamma $ }
\noLine
\UnaryInfC{$\nabla$}
\noLine
\UnaryInfC{$D\land E_1\strictif F$}
\doubleLine
\BinaryInfC{$D\strictif F$ }
\AxiomC{$\ldots$}
\noLine
\UnaryInfC{$ $}
 \AxiomC{$[D\strictif E_n]$ }
 \AxiomC{$\Gamma $ }
\noLine
\UnaryInfC{$\nabla$}
\noLine
\UnaryInfC{$D\land E_n\strictif F$}
\doubleLine
\BinaryInfC{$D\strictif F$ }
\QuaternaryInfC{$D\strictif F$}
\end{prooftree}
Again, double lines indicate a combination of applications of (trans) and
(conj2), while in the final step we use (disj2).

As for rule (disj2gen), suppose $B^1= B^1_1 \lor \dots \lor B^1_h$. We can apply
rule (disj2) to
\[
\Gamma \rhd_{\mathcal F_3} A^1 \strictif B^1
\]
and all premisses of the form
\[
\Gamma, A^1\strictif B^1_j, A^2 \strictif B^1_{h_2}, \dots, A^n\strictif B^1_{h_n} \rhd_{\mathcal F_3} C\strictif E,
\]
for $j=1, \dots, h$, obtaining, for all choices of indices $h_2, \dots, h_n$
such that $B^i_{h_i}$ is a disjunct of $B^i$, that
\[
\Gamma, A^2 \strictif B^2_{h_2},\ldots, A^n \strictif B^n_{h_n} \rhd_{\mathcal F_3} C \strictif E;
\]
In other words, we succeeded in eliminating $A^1\strictif B^1$ from the
premisses. In the same way, by applying (disj2) we can successively eliminate
$A^2 \strictif B^2, \dots, A^n \strictif B^n$, eventually deriving $\Gamma
\rhd_{\mathcal F_3} C \strictif E$, as desired.
\end{proof}

In order to prove the completeness of $\mathcal F_3$-derivability we need a
preliminary Lemma.

\begin{lemma}\label{lem:disjprop}
Suppose $\Gamma$ is a set of positive $\mathcal F_3$ s-formulas such that
$\Gamma \ntriangleright_{\mathcal F_3} C \strictif E$. Then there exists a
set of positive $\mathcal F_3$ s-formulas $\Delta$, closed under
$\rhd_{\mathcal F_3}$, such that:
\begin{itemize}
  \item $\Delta \supseteq \Gamma$;
  \item $\Delta \ntriangleright_{\mathcal F_3} C \strictif E,$;
  \item for all positive $\mathcal F_3$ s-formulas $A \strictif B \in
      \Delta$ there exists $i$ such that $A \strictif B_i \in \Delta$.
\end{itemize}
\end{lemma}
\begin{proof}
Without loss of generality, we may suppose that $\Gamma$ is closed under
$\rhd_{\mathcal F_3}$. Let $\{\alpha_1, \alpha_2, \dots\}$ be an enumeration
of the positive ${\mathcal F_3}$ s-formulas, with $\alpha_j = A^j \strictif
B^j$.

We claim that there exists a sequence $\Gamma_0 =\Gamma, \Gamma_1, \dots,
\Gamma_n, \ldots$ of sets of positive $\mathcal F_3$ s-formulas, each closed
under $\rhd_{\mathcal F_3}$, with the following properties:
\begin{itemize}
  \item $\Gamma_{n} \ntriangleright_{\mathcal F_3} C \strictif E$;
  \item if, for $j \leq n$, $\Gamma_{n} \rhd_{\mathcal F_3} \alpha_j$, then
      there exists $h$ such that $A^j \strictif B^j_h \in \Gamma_{n+1}$.
\end{itemize}

We start by setting $\Gamma_0 = \Gamma$. Suppose now we already defined
$\Gamma_n$ such that $\Gamma_{n} \ntriangleright_{\mathcal F_3} C \strictif
E$. Let $j_1, \dots j_h \leq n$ be the list of all indices up to $n$ such
that $\Gamma_{n} \rhd_{\mathcal F_3} \alpha_{j_i}$. Then there must exist a
choice of indices $h_{j_1}, \dots, h_{j_h}$ such that $B^{j_i}_{h_{j_i}}$ is
a disjunct of $B^{j_i}$, and
\[
\Gamma_{n}, A^{j_1} \strictif B^{j_1}_{h_{j_1}}, \dots, A^{j_n} \strictif B^{j_n}_{h_{j_h}}
\ntriangleright_{\mathcal F_3} C \strictif E.
\]
In fact, if this were not the case, using rule (disj2gen), we would obtain
that $\Gamma_{n} \rhd_{\mathcal F_3} C \strictif E$. We fix such $h_{j_1},
\dots, h_{j_h}$ and let $\Gamma_{n+1}$ be the closure of $\Gamma_{n} \cup
\{A^1 \strictif B^1_{h_{j_1}}, \dots, A^n \strictif B^n_{h_{j_h}}\}$ under
$\rhd_{\mathcal F_3}$. This proves the claim.

Finally, it is straightforward to check that $\Delta = \bigcup_n \Gamma_n$ has the
required properties.
\end{proof}

We split the proof of the completeness of $\rhd_{\mathcal F_3} $ into cases,
depending on the satisfiability of $\Gamma$ and on the type of the formula to
be derived. We start with:

\begin{lemma}\label{lem:positive}
Suppose $\Gamma$ is a satisfiable set of $\mathcal F_3$ s-formulas and $C
\strictif E$ is a positive $\mathcal F_3$ s-formula such that $\Gamma
\models_s C \strictif E$. Then $\Gamma \rhd_{\mathcal F_3} C \strictif E$.
\end{lemma}
\begin{proof}
We reason by contradiction. If $\Gamma \ntriangleright_{\mathcal F_3} C
\strictif E$ then $\Gamma^+ \ntriangleright_{\mathcal F_3} C\strictif E$,
either. By applying the previous Lemma to $\Gamma^+$ we find a set of
positive $\mathcal F_3$ s-formulas $\Delta \supseteq \Gamma^+$, closed under
${\rhd_{\mathcal F_3}}$, such that
\[
\Delta \ntriangleright_{\mathcal F_3} C \strictif E,
\]
and for all $\mathcal F_3$ s-formulas $A \strictif B$, if $A \strictif B \in
\Delta$ then there exists $i$ with $A \strictif B_i \in \Delta$.

Let $w$ be the valuation defined by setting, for each propositional variable
$X$:
\[
w(X)=
\begin{cases}
T & \hbox{if $C \strictif X \in \Delta$;}\\
F & \hbox{if $C \strictif X \notin \Delta$.}
\end{cases}
\]

We claim that $w(\Delta)= T$, and $w( C\strictif E)=F$.

If $B \strictif A \in \Delta$ and $w(B)=T$, then, since $B=B_1 \land \dots
\land B_n$, we have $w(B_i)=T$ for all $i$. By definition of $w$, for all $i$
it holds $C \strictif B_i \in \Delta$, and by rule ($r_2$) we obtain $C
\strictif A \in \Delta$. By the property of $\Delta$ there exists $i$ such
that $C\strictif A_i \in \Delta$. Hence $w(A_i)=T$ and therefore $w(A)=T$ as
well. This proves that $w(B \strictif A)=T$, for all $B \strictif A \in
\Delta$.

Let us now show that $w(C \strictif E)=F$. Since $w(C)=T$, it suffices to
prove that $w(E_i)=F$, for all $i$. If $w(E_i)=T$ for some $i$, then $C
\strictif E_i \in \Delta$ and $C \strictif E \in \Delta$ would follow by rule
(disj1).\smallskip

Having established the claim, we conclude the proof as follows. For all
negative $\mathcal F_3$ s-formulas $\alpha = A \not\strictif B \in \Gamma$,
let $v_\alpha$ be a valuation such that $v_\alpha(\Gamma)=T$, $v_\alpha(A)=T$
and $v_\alpha(B)=F$. Such a $v_\alpha$ exists, because by hypothesis $\Gamma$
is satisfiable. Then the frame $W = \{w\} \cup \{v_\alpha: \alpha \in
\Gamma^-\}$ is such that $W \models \Gamma$ and $W \nvDash C \strictif E$,
contradicting our hypothesis.
\end{proof}

Next, we consider the case in which $\Gamma$ is satisfiable, but the formula
to be derived is negative.

\begin{lemma}
Suppose $\Gamma$ is a satisfiable set of $\mathcal F_3$ s-formulas and $C
\not\strictif G$ is a negative $\mathcal F_3$ s-formula such that $\Gamma
\models_s C \not\strictif G$. Then $\Gamma \rhd_{\mathcal F_3} C
\not\strictif G$.
\end{lemma}
\begin{proof}
This proof follows the corresponding proof in \cite{Mummert} with minor
adjustments. We reason again by contradiction supposing (without loss of
generality) that $\Gamma$ closed under $\rhd_{\mathcal F_3}$ and $ C\not
\strictif G \not \in \Gamma$. For any $\alpha= D \not\strictif E \in
\Gamma^-$, we will find a valuation $w_\alpha$ with $w_\alpha(\Gamma^+)=T$,
$w_\alpha(D)=T$ and $w_\alpha(E)=F$, and either $w_\alpha( C)=F$ or
$w_\alpha(G)=T$. Once this is done, we may set
\[
W= \{w_\alpha : \alpha \in \Gamma^-\}
\]
and find a contradiction, since $W$ is a frame satisfying $\Gamma$ but
failing to satisfy $C \not\strictif G$.

Fix $\alpha=D\not \strictif E\in \Gamma^-$. Since $\Gamma$ is satisfiable,
there exists a valuation $w$ with $w(\Gamma^+)=T$, $w(D)=T$, and $w(E)=F$. In
order to find $w_\alpha$ we may suppose that all the valuations $w$ with
these properties satisfy also $w(C)=T$ (otherwise we may choose such a $w$
for $w_\alpha$). Consider the set of positive s-formulas $\Gamma^+ \cup \{C
\strictif G\}$. Then
\[
\Gamma^+ \cup \{C \strictif G\} \ntriangleright_{\mathcal F_3} D \strictif E,
\]
otherwise, since $\Gamma\rhd_{\mathcal F_3} D \not \strictif E$, we would
have $\Gamma \rhd_{\mathcal F_3} C \not\strictif G$ be the (neg) rule. By
Lemma \ref{lem:disjprop} there exists a set of positive formulas
$\Delta\supseteq \Gamma^+\cup\{C\strictif G\}$, closed under $\rhd_{\mathcal
F_3}$, such that $\Delta \ntriangleright_{\mathcal F_3} D \strictif E$, and
for all $A,B$, if $A \strictif B \in \Delta$ then there exists $i$ with $A
\strictif B_i \in \Delta$. We claim that $D \strictif C_i \in \Delta$ for
every $i$. To see this, we consider the valuation $w$ defined as
\[
w(X)=
\begin{cases}
T & \hbox{if $D \strictif X \in \Delta$;}\\
F & \hbox{if $D \strictif X \notin \Delta$.}
\end{cases}
\]
As in Lemma \ref{lem:positive}, it is not difficult to check that
$w(\Delta)=T$, and $w (D \strictif E)=F$. By the previous hypothesis, we have
$w(C)=T$, that is, $w(C_i)=T$ for all $i$. By definition of $w$ this implies
$D \strictif C_i \in \Delta$.

Next, consider the valuation $v_i$ defined as
\[
v_i(X)=
\begin{cases}
T & \hbox{if $D\land G_i \strictif X \in \Delta$;}\\
F & \hbox{if $D\land G_i \strictif X \notin \Delta$.}
\end{cases}
\]
We claim that there exists $i$ with $v_i(E)=F$. Otherwise, we have
$v_i(E)=T$, for all $i$. This means that for all $i$ there exists $j$ with
$v_i(E_j)=T$, that is, by definition of $v_i$, $D \land G_i \strictif E_j \in
\Delta$. It follows that, for all $i$, $D \land G_i \strictif E \in \Delta$.
Consider the following natural deduction, which uses first ($r_2$) and then
($r_3$);
\begin{prooftree}
 \AxiomC{$\Delta$ }
\noLine
\UnaryInfC{$\nabla$}
\noLine
\UnaryInfC{$C\strictif G$ }
 \AxiomC{$\Delta$ }
\noLine
\UnaryInfC{$\nabla$}
\noLine
\UnaryInfC{$D\strictif C_1$}
\AxiomC{$\ldots$}
\noLine
\UnaryInfC{$ $}
 \AxiomC{$\Delta$ }
\noLine
\UnaryInfC{$\nabla$}
\noLine
\UnaryInfC{$D\strictif C_k$}
\QuaternaryInfC{$D\strictif G$}
\AxiomC{$\Delta$ }
\noLine
\UnaryInfC{$\nabla$}
\noLine
\UnaryInfC{$D\land G_1\strictif E$}
\AxiomC{$\ldots$}
\noLine
\UnaryInfC{$ $}
\AxiomC{$\Delta$ }
\noLine
\UnaryInfC{$\nabla$}
\noLine
\UnaryInfC{$D\land G_n\strictif E$}
\QuaternaryInfC{$D\strictif E$}
\end{prooftree}
This contradicts $\Delta \ntriangleright D \strictif E$.

Thus we can pick $i$ such that $v_i(E)=F$. We have $v_i(D)=T$, $v_i(E)=F$,
and $v_i(G)=T$, since $D \land G_i \strictif G_i \in \Delta$ and $G$ is a
disjunction. Moreover, as before, $v_i(\Delta)=T$: if $A\strictif B\in
\Delta$ and $v_i(A)=T$, then $ D\land G_i \strictif A_j\in \Delta$, for all
$j$. By rule ($r_2$) we obtain $D\land G_i \strictif B\in \Delta$, and by the
properties of $\Delta$ there exists $h$ with $D\land G_i \strictif B_h\in
\Delta$; hence, $v_i(B_h)=T$, and $v_i(B)=T$. It follows that
$v_i(\Gamma^+)=T$, and we may choose such a $v_i$ as $w_\alpha$, finishing
the proof.
\end{proof}

The two previous results prove that, if $\Gamma$ is a satisfiable set of
${\mathcal F}_3$ s-formulas, then for any ${\mathcal F}_3$ s-formula $\alpha$
such that $\Gamma \models_s \alpha$ we have $\Gamma \rhd_{\mathcal F_3}
\alpha$.

To finish the completeness proof for $\rhd_{\mathcal F_3}$, we still have to
consider the case when $\Gamma$ is unsatisfiable. In this case we have to
prove that $\Gamma \rhd_{\mathcal F_3} \alpha$, for any ${\mathcal F}_3$
s-formula $\alpha$, and we  may repeat the proof of Lemma
\ref{lem:unsat}. Hence:

\begin{lemma}
If $\Gamma$ is unsatisfiable, then for any ${\mathcal F}_3$ s-formula
$\alpha$ we have $\Gamma \rhd_{\mathcal F_3} \alpha$.
\end{lemma}

Putting all results of this section together, we obtain:
\begin{theorem}
If $\Gamma$ is a set of  ${\mathcal F}_3$ s-formulas and $\alpha$ is a
${\mathcal F}_3$ s-formula, then
\[
\Gamma \models_s \alpha \qquad \Leftrightarrow \qquad \Gamma \rhd_{\mathcal F_3} \alpha.
\]
\end{theorem}

\section{${\mathcal F_2}$ and Prolog}\label{sec:prolog}

In this section we show how standard Prolog may be used to deal with logical
consequence in $\mathcal{F}_2$. Since some readers might be unfamiliar with
Prolog, we recall here the basic constructs of this programming language
(restricting ourselves to the propositional setting), following \cite{Nerode}
(see \S I.10, and especially Definition 10.4).

Propositional Prolog deals with \emph{Horn clauses} (finite sets of literals
containing at most one positive literal), thought as disjunctions of their
elements. When the Horn clause contains (exactly) one positive literal $\{Y,
\neg X_1, \dots, \neg X_n\}$ it is a \emph{program clause} and we write $Y
\pl X_1, \dots, X_n$. If $n>0$ we think that the program clause is
representing $X_1 \land \dots \land X_n \rightarrow Y$ and we call it a
\emph{rule}. If in the program clause we have $n=0$ it is a \textit{fact} and
we write $Y \pl$. If the Horn clause has only negative literals $\{\neg X_1,
\dots, \neg X_n\}$ we call it a \emph{goal} and write $\pl X_1, \dots, X_n$.
A \emph{Prolog program} is a set of program clauses.

The typical situation is that we are given a Prolog program, and we want to
know whether a conjunction of facts $Y_1, \dots, Y_k$ is logical consequence
of the given facts and rules. To this end we add the goal $\{\neg Y_1, \dots,
\neg Y_k\}$ to the program and ask whether the resulting set of Horn clauses
is unsatisfiable. This is the case if and only if applying the resolution
rule repeatedly to the elements of the set starting with the goal we obtain
the empty clause. Prolog works by searching all possible ways of applying the
resolution rule with these constraints: if the search succeeds we have a
\emph{refutation} of the goal from the program.

We can now go back to our study of the $\mathcal{F}_2$ fragment of s-logic.

\begin{definition}
Given a set $\Gamma$ of $\mathcal{F}_2$ s-formulas, define $Prolog(\Gamma^+)$
to be the following Prolog program:
\[
Prolog(\Gamma^+) = \{Z \pl A_1, \dots, A_n \mid A_1 \land \ldots \land A_n \strictif Z \in \Gamma^+\}.
\]
\end{definition}

We have:

\begin{lemma}\label{lemma:prolog}
Let $\Gamma$ be a set of $\mathcal{F}_2$ s-formulas and $A \strictif Y$ be a
$\mathcal{F}_2$ s-formula, where $A = A_1 \land \dots \land A_n$.
\begin{enumerate}[(i)]
 \item $\Gamma \models_s A \strictif Y$ if and only there is a refutation
     of the goal $\pl Y$ from the Prolog program
   \[
    Prolog(\Gamma^+) \cup \{A_1\pl, \dots, A_n\pl\};
    \]
 \item $\Gamma \models_s A \not\strictif Y$ if and only if there exists
     $Z_1 \land \dots \land Z_n \not\strictif W \in
   \Gamma^-$ and a refutation of the goal $\pl W$ from the Prolog program
   \[
   Prolog(\Gamma^+) \cup \{Y \pl A_1, \ldots, A_n, Z_1\pl, \dots, Z_n\pl\}.
   \]
\end{enumerate}
\end{lemma}
\begin{proof}
\begin{enumerate}[(i)]
 \item From Lemma \ref{lemma:logcons}.i we have that $\Gamma \models A
     \strictif Y$ if and only if $\Gamma^+_{prop}, A \models Y$. Since
     $\Gamma$ is a set of $\mathcal{F}_2$-formulas, the elements in
     $\Gamma^+_{prop}$ are (essentially) rules, while $A$ is equivalent to
     the conjunction of the facts $A_1\pl, \dots, A_n\pl$. Since $Y$ is a
     positive literal, the equivalence follows from the completeness of
     Propositional Prolog.
 \item From Lemma \ref{lemma:logcons}.ii we have that $\Gamma \models A
     \not\strictif Y$ if and only if there exists $Z_1 \land \dots \land
     Z_n \not\strictif W \in \Gamma^-$ such that
\[
\Gamma^+_{prop}, A \rightarrow Y, Z_1 \land \dots \land Z_n \models W.
\]
As before, the equivalence follows from interpreting this logical
consequence in terms of Prolog and applying the completeness of
Propositional Prolog.\qedhere
\end{enumerate}
\end{proof}

Lemma \ref{lemma:prolog} suggests an efficient way of checking logical
consequence between $\mathcal{F}_2$ s-formulas based on a well-known
programming language such as Prolog, and actually only for the special case
of goals consisting of a single literal.

\bibliographystyle{alpha}
\bibliography{logiczoo}

\end{document}